\documentclass[a4paper,oneside,leqno]{amsart}

\usepackage{geometry}
\usepackage[utf8]{inputenc}
\usepackage[T1, T2A]{fontenc}
\usepackage[bulgarian, american]{babel}

\usepackage{amsmath}
\usepackage{amsthm}
\usepackage{amssymb}
\usepackage{hyperref}
\usepackage{caption}
\usepackage{subcaption}
\usepackage{graphicx}

\newtheorem{theorem}{Theorem}
\newtheorem*{maintheorem}{Main Theorem}
\newtheorem{lemma}[theorem]{Lemma}
\newtheorem{proposition}[theorem]{Proposition}

\newtheorem{corollary}[theorem]{Corollary}

\newtheorem{algorithm}[theorem]{Algorithm}
\newtheorem{procedure}[theorem]{Procedure}

\numberwithin{equation}{section}
\numberwithin{theorem}{section}

\newcommand{\abs}[1]{\vert{#1}\vert}
\newcommand{\set}[1]{\left\{#1\right\}}

\newcommand{\arrowsv}[0]{\overset{v}{\rightarrow}}
\newcommand{\arrowse}[0]{\overset{e}{\rightarrow}}
\newcommand{\mH}[0]{\mathcal{H}}

\DeclareMathOperator{\V}{V}
\DeclareMathOperator{\E}{E}

\begin{document}
	
% -----------------------------------------------------------
	
\title{New bounds on the vertex Folkman number $F_v(2, 2, 2, 3; 4)$}

\author{Aleksandar Bikov}

%\address{Faculty of Mathematics and Informatics, University of Sofia, Bulgaria}

%\email[Corresponding author]{asbikov@fmi.uni-sofia.bg}



\dedicatory{}

\begin{abstract}
For a graph $G$ the expression $G \overset{v}{\rightarrow} (a_1, ..., a_s)$ means that for every coloring of the vertices of $G$ in $s$ colors there exists $i \in \set{1, ..., s}$ such that there is a monochromatic $a_i$-clique of color $i$. The vertex Folkman number $F_v(a_1, ..., a_s; q)$ is defined as
$$F_v(a_1, ..., a_s; q) = \min\{\vert\V(G)\vert : G \overset{v}{\rightarrow} (a_1, ..., a_s) \mbox{ and } K_q \not\subseteq G\}.$$
In this paper we improve the known bounds on the number $F_v(2, 2, 2, 3; 4)$ by proving with the help of a computer that $20 \leq F_v(2, 2, 2, 3; 4) \leq 22$.

\end{abstract}

\maketitle

% -----------------------------------------------------------

\section{Introduction}

Only finite, non-oriented graphs without loops and multiple edges are considered in this paper. The vertex set and the edge set of a graph $G$ are denoted by $\V(G)$ and $\E(G)$ respectively. $K_n$ stands for a complete graph or a clique on $n$ vertices, $\omega(G)$ for the clique number of $G$, and $\alpha(G)$ for the independence number of $G$. For two given graphs $G_1$ and $G_2$ we denote by $G_1+G_2$ the graph $G$ obtained by making every vertex of $G_1$ adjacent to every vertex of $G_2$. All undefined terms can be found in \cite{Wes01}.

The Ramsey number $R(p, q)$ is defined as the least positive integer $n$ such that every graph on $n$ vertices contains either a $p$-clique or an independent set of size $q$. More information on the Ramsey numbers used in this paper can be found in \cite{Rad14}.

Let $a_1, ..., a_s$ be positive integers. The expression $G \overset{v}{\rightarrow} (a_1, ..., a_s)$ means that for every coloring of $\V(G)$ in $s$ colors ($s$-coloring) there exists $i \in \set{1, ..., s}$ such that there is a monochromatic $a_i$-clique of color $i$. 

Define:

$\mH_v(a_1, ..., a_s; q) = \set{ G : G \overset{v}{\rightarrow} (a_1, ..., a_s) \mbox{ and } \omega(G) < q }.$

$\mH_v(a_1, ..., a_s; q; n) = \set{ G : G \in \mH_v(a_1, ..., a_s; q) \mbox{ and } \abs{\V(G)} = n }.$

The vertex Folkman number $F_v(a_1, ..., a_s; q)$ is defined as

$F_v(a_1, ..., a_s; q) = \min\set{\abs{\V(G)} : G \in \mH_v(a_1, ..., a_s; q)}$.

Folkman proves in \cite{Fol70} that
\begin{equation}
\label{equation: F_v(a_1, ..., a_s; q) exists}
F_v(a_1, ..., a_s; q) \mbox{ exists } \Leftrightarrow q > \max\set{a_1, ..., a_s}.
\end{equation}
Obviously $F_v(a_1, ..., a_s; q)$ is a symmetric function of $a_1, ..., a_s$, and if $a_i = 1$, then
\begin{equation*}
F_v(a_1, ..., a_s; q) = F_v(a_1, ..., a_{i-1}, a_{i+1}, ..., a_s; q).
\end{equation*}
Therefore, it is enough to consider only such Folkman numbers $F_v(a_1, ..., a_s; q)$ for which $2 \leq a_1 \leq ... \leq a_s$. In the case $a_r = 2$, for convenience, instead of $G \overset{v}{\rightarrow} (\underbrace{2, ..., 2}_r, a_{r + 1}, ..., a_s)$ we can write $G \overset{v}{\rightarrow} (2_r, a_{r + 1}, ..., a_s)$.

In \cite{LU96} Luczak and Urba{\'n}ski defined for arbitrary positive integers $a_1, ..., a_s$:
\begin{equation*}
\label{equation: m and p}
m(a_1, ..., a_s) = m = \sum\limits_{i=1}^s (a_i - 1) + 1 \quad \mbox{ and } \quad p = \max\set{a_1, ..., a_s}.
\end{equation*}

If $q \geq m + 1$, then $F_v(a_1, ..., a_s; q) = m$, since $K_m \overset{v}{\rightarrow} (a_1, ..., a_s)$ and $K_{m - 1} \overset{v}{\nrightarrow} (a_1, ..., a_s)$.

If $q = m$, by (\ref{equation: F_v(a_1, ..., a_s; q) exists}), $F_v(a_1, ..., a_s; m)$ exists if and only if $m \geq p + 1$. If $m \geq p + 1$, then $F_v(a_1, ..., a_s; m) = m + p$ and $K_{m - p - 1} + \overline{C}_{2p + 1}$ is the only $(m + p)$-vertex graph in $\mH_v(a_1, ..., a_s; m)$, \cite{LU96}, \cite{LRU01}.

In the case $q = m - 1$, thanks to the contributions of different authors, the exact values of all the numbers $F_v(a_1, ..., a_s; m - 1)$ for which $p \leq 6$ are known. For more information about these numbers see \cite{BN15a} and \cite{BN15b}.

Not much is known about the numbers $F_v(a_1, ..., a_s; q)$ when $q \leq m - 2$. In the case $q = m - 2$ we know all the numbers where $p = 2$, i.e. all the numbers of the form $F_v(2_r; r - 1)$. Nenov proved in \cite{Nen83} that $F_v(2_r; r - 1) = r + 7$ if $r \geq 8$, and later in \cite{Nen07} he proved that the same equality holds when $r \geq 6$. Other proof of the same equality was given in \cite{Nen09}. By a computer aided research, Jensen and Royle \cite{JR95} obtained the result $F_v(2_4; 3) = 22$. The last remaining number of this form is $F_v(2_5; 4) = 16$. The upper bound was proved by Nenov in \cite{Nen07} and the lower bound was proved by Lathrop and Radziszowski in \cite{LR11} with the help of a computer. 

The exact values of the numbers $F_v(a_1, ..., a_s; m - 2)$ for which $p \geq 3$ are unknown and no good general bounds have been obtained. According to (\ref{equation: F_v(a_1, ..., a_s; q) exists}), $F_v(a_1, ..., a_s; m - 2)$ exists if and only if $m \geq p + 3$. If $p = 3$, it is easy to see that in the border case $m = 6$ there are only two numbers of the form $F_v(a_1, ..., a_s; m - 2)$, namely $F_v(2, 2, 2, 3; 4)$ and $F_v(2, 3, 3; 4)$. Since $G \arrowsv (2, 3, 3) \Rightarrow G \arrowsv (2, 2, 2, 3)$, it follows that
\begin{equation*}
\label{equation: F_v(2, 2, 2, 3; 4) leq F_v(2, 3, 3; 4)}
F_v(2, 2, 2, 3; 4) \leq F_v(2, 3, 3; 4).
\end{equation*}
Computing and obtaining bounds on the numbers $F_v(2, 2, 2, 3; 4)$ and $F_v(2, 3, 3; 4)$ has an important role in relation to computing and obtaining bounds on the other numbers of the form $F_v(a_1, ..., a_s; m - 2)$ for which $p = 3$. For example, it is easy to prove that if $G \arrowsv (a_1, .., a_s)$, then $K_t + G \arrowsv (2_t, a_1, ..., a_s)$, and therefore
\begin{equation}
\label{equation: F_v(2_r, 3; r + 1) leq F_v(2, 2, 2, 3; 4) + r - 3}
F_v(2_r, 3; r + 1) \leq F_v(2, 2, 2, 3; 4) + r - 3, r \geq 3.
\end{equation}
Also, according to Theorem 4.5 and Theorem 7.2 ($m_0 = 6, p = 3, q = 4$) from \cite{BN15a}
\begin{equation}
\label{equation: F_v(a_1, ..., a_s; m - 2) leq F_v(2, 3, 3; 4) + m - 6}
F_v(a_1, ..., a_s; m - 2) \leq F_v(2, 3, 3; 4) + m - 6, \mbox{ if } m \geq 6 \mbox{ and } \max\set{a_1, ..., a_s} = 3.
\end{equation}
Shao, Xu and Luo \cite{SXL09} proved in 2009 that
\begin{equation*}
18 \leq F_v(2, 2, 2, 3; 4) \leq F_v(2, 3, 3; 4) \leq 30.
\end{equation*}
In 2011 Shao, Liang, He and Xu \cite{SLHX11} raised the lower bound on $F_v(2, 3, 3; 4)$ to 19. In 2016 Nenov and the current author \cite{BN16} obtained the bounds
\begin{equation}
\label{equation: 20 leq F_v(2, 3, 3; 4) leq 24}
20 \leq F_v(2, 3, 3; 4) \leq 24.
\end{equation}
In this paper, we improve the bounds on the number $F_v(2, 2, 2, 3; 4)$ by proving the following
\begin{maintheorem}
\label{maintheorem: 20 leq F_v(2, 2, 2, 3; 4) leq 22}
$20 \leq F_v(2, 2, 2, 3; 4) \leq 22$.
\end{maintheorem}
We prove the lower bound from the Main Theorem in Section 2 and the upper bound in Section 3. Substituting $F_v(2, 2, 2, 3; 4)$ with its new upper bound in (\ref{equation: F_v(2_r, 3; r + 1) leq F_v(2, 2, 2, 3; 4) + r - 3}) gives the following
\begin{corollary}
\label{corollary: F_v(2_r, 3; r + 1) leq r + 19}
$F_v(2_r, 3; r + 1) \leq r + 19, r \geq 3$.
\end{corollary}
Let us also note the following corollary from  (\ref{equation: 20 leq F_v(2, 3, 3; 4) leq 24}) and (\ref{equation: F_v(a_1, ..., a_s; m - 2) leq F_v(2, 3, 3; 4) + m - 6}), which gives a general upper bound on the numbers $F_v(a_1, ..., a_s; m - 2)$ for which $p = 3$
\begin{corollary}
\label{corollary: F_v(a_1, ..., a_s; m - 2) leq m + 18}
$F_v(a_1, ..., a_s; m - 2) \leq m + 18, \mbox{ if } m \geq 6 \mbox{ and } \max\set{a_1, ..., a_s} = 3$.
\end{corollary}

\section{Proof of the lower bound $F_v(2, 2, 2, 3; 4) \geq 20$}

We will prove with the help of a computer that $\mH_v(2, 2, 2, 3; 4; 19) = \emptyset$. Even for a very powerful machine it is practically impossible to accomplish this task by an exhaustive search through the set of all 19-vertex graphs. Instead, the computational approach is based on the following useful fact, which is easy to prove (see Proposition 2.2 from \cite{BN15b}):

\begin{proposition}
\label{proposition: H}
Let $G \in \mH_v(a_1, ..., a_s; q; n)$, $A \subseteq \V(G)$ be an independent set of vertices of $G$ and $H = G - A$. Then, $H \in \mH_v(a_1 - 1, a_2, ..., a_s; q; n - |A|)$.
\end{proposition}

According to Proposition \ref{proposition: H}, if $G \arrowsv (2, 3, 3)$ and the graph $H$ is obtained by removing an independent set of vertices from $G$, then $H \arrowsv (3, 3)$. Shao, Liang, He and Xu showed in \cite{SLHX11} that if $G \in \mH_v(2, 3, 3; 4; 18)$, then $G$ can be obtained by adding 4 independent vertices to 111 of the 153 graphs in $\mH_v(3, 3; 4; 14)$, which were obtained in \cite{PRU99}. With the help of a computer, the authors showed that such extension does not produce any graphs in $\mH_v(2, 3, 3; 4; 18)$ and proved the lower bound $F_v(2, 3, 3; 4) \geq 19$. Similarly, they showed that if $G \in \mH_v(2, 2, 2, 3; 4; 18)$, then $G$ can be obtained by adding 4 independent vertices to 12064 of the 12227 graphs in $\mH_v(2, 2, 3; 4; 14)$, which were obtained in \cite{CR06}. Computing the bound $F_v(2, 2, 2, 3; 4) \geq 19$ is harder because of the larger number of graphs that have to be extended.

In \cite{BN16} Nenov and the current author obtained with the help of a computer all 2081214 graphs in $\mH_v(3, 3; 4; 15)$ and proved the lower bound $F_v(2, 3, 3; 4) \geq 20$. In this section we prove the stronger inequality $F_v(2, 2, 2, 3; 4) \geq 20$, which requires much more computational effort. To obtain this result we modify Algorithm 2.4 (Algorithm 2.7 in submitted version) from \cite{BN16}.

We will need the following auxiliary definitions and propositions. The graph $G$ is called a maximal graph in $\mH_v(2_r, 3; 4; n)$ if $\omega(G + e) = 4$ for every $e \in \E(\overline{G})$. We say that $G$ is a $(+K_3)$-graph if $G + e$ contains a new $3$-clique for every $e \in \E(\overline{G})$. Suppose that $G$ is a maximal graph in $\mH_v(2_r, 3; 4; n)$ and $\alpha(G) = s$. Let $A \subseteq \V(G)$ be an independent set of vertices of $G$, $\abs{A} = s$ and $H = G - A$. By Proposition \ref{proposition: H}, we have $H \in \mH_v(2_{r - 1}, 3; 4; n - s)$. From $\omega(G + e) = 4, \forall e \in \E(\overline{G})$ it follows that $H$ is a $(+K_3)$-graph, and obviously, from $\alpha(G) = s$ it follows that $\alpha(H) \leq s$. Thus, we proved the following
\begin{proposition}
\label{proposition: extend input}
Let $G$ be a maximal graph in $\mH_v(2_r, 3; 4; n)$ and $\alpha(G) = s$. Let $A \subseteq \V(G)$ be an independent set of vertices of $G$, $\abs{A} = s$ and $H = G - A$. Then $H$ is a $(+K_3)$ graph in $\mH_v(2_{r - 1}, 3; 4; n - s)$ with independence number not greater than s.
\end{proposition}

The graph $G$ is called a Sperner graph if $N_G(u) \subseteq N_G(v)$ for some pair of vertices $u, v \in \V(G)$. If $G \in \mH_v(2_r, 3; 4; n)$ is a Sperner graph and $N_G(u) \subseteq N_G(v)$, then $G - u \in \mH_v(2_r, 3; 4; n - 1)$. In the special case when $G$ is a maximal Sperner graph in $\mH_v(2_r, 3; 4; n)$, from $N_G(u) \subseteq N_G(v)$ we derive $N_G(u) = N_G(v)$, and it follows that $G - u$ is a maximal graph in $\mH_v(2_r, 3; 4; n - 1)$. Thus, we proved the following

\begin{proposition}
\label{proposition: maximal Sperner graphs}
Every maximal Sperner graph in $\mH_v(2_r, 3; 4; n)$ is obtained by duplicating a vertex in some of the maximal graphs in $\mH_v(2_r, 3; 4; n - 1)$.
\end{proposition}

All maximal non-Sperner graphs with independence number $s$ in $\mH_v(2_r, 3; 4; n)$ are obtained very efficiently with the following algorithm which is a slightly modified variant of Algorithm 2.4 (Algorithm 2.7 in the submitted version) from \cite{BN16} (the set $\mathcal{L}(n; p; s)$ is replaced with the set $\mH_v(2_r, 3; 4; n)$ and the last check $K_p + G \arrowse (3, 3)$ is replaced with $G \arrowsv (2_r, 3)$). As noted in \cite{BN16}, this algorithm can be used not only to prove $F_e(3, 3; 4) \geq 20$, but also to obtain other results. At the end of \cite{BN16}, one such example is given by proving $F_v(2, 3, 3; 4) \geq 20$. In this paper we use the following modification of this algorithm to obtain another new result ($F_v(2, 2, 2, 3; 4) \geq 20$).

\begin{algorithm}
\label{algorithm: extend}
Finding all maximal non-Sperner graphs with independence number $s$ in $\mH_v(2_r, 3; 4; n)$, for fixed $n$, $r$ and $s$.

\emph{1.} The input of the algorithm is the set $\mathcal{A}$ of all $(+K_3)$-graphs in $\mH_v(2_{r - 1}, 3; 4; n - s)$ with independence number not greater than $s$. The obtained graphs will be output in the set $\mathcal{B}$, let $\mathcal{B} = \emptyset$.

\emph{2.} For each graph $H \in \mathcal{A}$:

\emph{2.1.} Find the family $\mathcal{M}(H) = \set{M_1, ..., M_t}$ of all maximal $K_3$-free subsets of $\V(H)$.

\emph{2.2.} Find all $s$-element subsets $N = \set{M_{i_1}, M_{i_2}, ..., M_{i_s}}$ of $\mathcal{M}(H)$, which fulfill the conditions:

(a) $M_{i_j} \neq N_H(v)$ for every $v \in \V(H)$ and for every $M_{i_j} \in N$.

(b) $K_2 \subseteq H[M_{i_j} \cap M_{i_k}]$ for every $M_{i_j}, M_{i_k} \in N$.

(c) $\alpha(H - \bigcup_{M_{i_j} \in N'} M_{i_j}) \leq s - \abs{N'}$ for every $N' \subseteq N$.

\emph{2.3.} For each $s$-element subset $N = \set{M_{i_1}, M_{i_2}, ..., M_{i_s}}$ of $\mathcal{M}(H)$ found in step 2.2 construct the graph $G = G(N)$ by adding new independent vertices $v_1, v_2, ..., v_s$ to $\V(H)$ such that $N_G(v_j) = M_{i_j}, j = 1, ..., s$. If $G$ is not a Sperner graph and $\omega(G + e) = 4, \forall e \in \E(\overline{G})$, then add $G$ to $\mathcal{B}$.

\emph{3.} Remove the isomorph copies of graphs from $\mathcal{B}$.

\emph{4.} Remove from $\mathcal{B}$ all graphs $G$ for which $G \not\arrowsv (2_r, 3)$.
\end{algorithm}

We will prove the correctness of Algorithm \ref{algorithm: extend} with the help of the following
\begin{lemma}
\label{lemma: extend}
\cite{BN16}
After the execution of step 2.3 of Algorithm \ref{algorithm: extend}, the obtained set $\mathcal{B}$ coincides with the set of all maximal $K_4$-free non-Sperner graphs $G$ with $\alpha(G) = s$ which have an independent set of vertices $A \subseteq \V(G), \abs{A} = s$ such that $G - A \in \mathcal{A}$.
\end{lemma}

\begin{theorem}
\label{theorem: extend}
After the execution of Algorithm \ref{algorithm: extend}, the obtained set $\mathcal{B}$ coincides with the set of all maximal non-Sperner graphs with independence number $s$ in $\mH_v(2_r, 3; 4; n)$.
\end{theorem}

\begin{proof}
Suppose that, after the execution of Algorithm \ref{algorithm: extend}, $G \in \mathcal{B}$. According to Lemma \ref{lemma: extend}, $G$ is a maximal $K_4$-free non-Sperner graph with independence number $s$. From step 4 it follows that $G \in \mH_v(2_r, 3; 4; n)$. 
	
Conversely, let $G$ be an arbitrary maximal non-Sperner graph with independence number $s$ in $\mH_v(2_r, 3; 4; n)$. Let $A \subseteq \V(G)$ be an independent set of vertices of $G$, $\abs{A} = s$ and $H = G - A$. According to Proposition \ref{proposition: extend input}, $H \in \mathcal{A}$, and from Lemma \ref{lemma: extend} it follows that, after the execution of step 2.3, $G \in \mathcal{B}$. Clearly, after step 4, $G$ remains in $\mathcal{B}$.
\end{proof}

We shall note the important role of the \emph{nauty} programs \cite{MP13} in this work. We use them for fast generation of non-isomorphic graphs and for graph isomorph rejection.

We performed various tests to check the correctness of the implementation of Algorithm \ref{algorithm: extend}. For example, we used the algorithm to reproduce all 12227 graphs in $\mH_v(2, 2, 3; 4; 14)$, which were obtained in \cite{CR06}. A lot of different tests of step 2 of the algorithm were performed in \cite{BN16}. All computations were completed in about a week on a personal computer. 

\subsection*{Proof of the bound $F_v(2, 2, 2, 3; 4) \geq 19$}

It is enough to prove that $\mH_v(2, 2, 2, 3; 4; 18) = \emptyset$. From $R(4, 4) = 18$ it follows that there are no graphs with independence number less than 4 in $\mH_v(2, 2, 2, 3; 4; 18)$, and from Proposition \ref{proposition: H} and $F_v(2, 2, 3; 4) = 14$ \cite{CR06} we derive that there are no graphs with independence number more than 4 in this set. Since $F_v(2, 2, 2, 3; 4) \geq 18$, there are no Sperner graphs in $\mH_v(2, 2, 2, 3; 4; 18)$. It remains to be proved that there are no non-Sperner graphs with independence number 4 in $\mH_v(2, 2, 2, 3; 4; 18)$. We apply Algorithm \ref{algorithm: extend} ($n = 18, r = 3, s = 4$). In step 1, $\mathcal{A}$ is the set of all 11844 $(+K_3)$-graphs with independence number less than 5 in $\mH_v(2, 2, 3; 4; 14)$. Let us remind, that all 12227 graphs in $\mH_v(2, 2, 3; 4; 14)$ were obtained in \cite{CR06}. After the execution of step 3, 130923 graphs remain in the set $\mathcal{B}$. None of these graphs belong to $\mH_v(2, 2, 2, 3; 4)$, and therefore after step 4 we have $\mathcal{B} = \emptyset$. Now, from Theorem \ref{theorem: extend} we conclude that there are no non-Sperner graphs with independence number 4 in $\mH_v(2, 2, 2, 3; 4; 18)$, which finishes the proof.

\subsection*{Proof of the bound $F_v(2, 2, 2, 3; 4) \geq 20$}

It is enough to prove that $\mH_v(2, 2, 2, 3; 4; 19) = \emptyset$. Again, from $R(4, 4) = 18$ it follows that there are no graphs with independence number less than 4 in $\mH_v(2, 2, 2, 3; 4; 18)$, and from Proposition \ref{proposition: H} and $F_v(2, 2, 3; 4) = 14$ \cite{CR06} we derive that there are no graphs with independence number more than 5 in this set. Since $F_v(2, 2, 2, 3; 4) \geq 19$, there are no Sperner graphs in $\mH_v(2, 2, 2, 3; 4; 19)$.

We use Algorithm \ref{algorithm: extend} ($n = 19, r = 3, s = 5$) to prove that there are no non-Sperner graphs with independence number 5 in $\mH_v(2, 2, 2, 3; 4; 19)$. In step 1, out of the 12227 graphs in $\mH_v(2, 2, 3; 4; 14)$ from \cite{CR06}, there are 11989 $(+K_3)$-graphs with independence number less than 6. After the execution of step 3, 27433657 graphs remain in the set $\mathcal{B}$. Since none of these graphs are in $\mH_v(2, 2, 2, 3; 4)$, after step 4 we have $\mathcal{B} = \emptyset$. Now, from Theorem \ref{theorem: extend} we conclude that there are no non-Sperner graphs with independence number 5 in $\mH_v(2, 2, 2, 3; 4; 19)$.

It remains to be proved that there are no non-Sperner graphs with independence number 4 in $\mH_v(2, 2, 2, 3; 4; 19)$. By an exhaustive computer search we find all 3423631 $(+K_3)$-graphs in $\mH_v(2, 3; 4; 11)$ with independence number not greater than 4. We apply Algorithm \ref{algorithm: extend} ($n = 15, r = 2, s = 4$) to obtain all 165614 maximal non-Sperner graphs in $\mH_v(2, 2, 3; 4; 15)$ with independence number 4. According to Proposition \ref{proposition: maximal Sperner graphs}, all maximal Sperner graphs in $\mH_v(2, 2, 3; 4; 15)$ are obtained by duplicating a vertex in the maximal graphs in $\mH_v(2, 2, 3; 4; 14)$. In this way, we find all 4603 maximal Sperner graphs with independence number 4 in $\mH_v(2, 2, 3; 4; 15)$. There are exactly 640 15-vertex $K_4$-free graphs with independence number less than 4, which are available on \cite{McK_r}. Among them, 35 are maximal graphs in $\mH_v(2, 2, 3; 4; 15)$. Thus, we found all 170252 maximal graphs in $\mH_v(2, 2, 3; 4; 15)$ with independence number less than 5. By removing edges from these graphs, we find all 68783156 $(+K_3)$-graphs in $\mH_v(2, 2, 3; 4; 15)$ with independence number less than 5. Now, we apply Algorithm \ref{algorithm: extend} ($n = 19, r = 3, s = 4$) to obtain all maximal non-Sperner graphs with independence number 4 in $\mH_v(2, 2, 2, 3; 4; 19)$. After the execution of step 3, 347307340 graphs remain in the set $\mathcal{B}$. Similarly to the previous case, none of these graphs are in $\mH_v(2, 2, 2, 3; 4)$, and after step 4 we have $\mathcal{B} = \emptyset$. By Theorem \ref{theorem: extend}, it follows that there are no non-Sperner graphs with independence number 4 in $\mH_v(2, 2, 2, 3; 4; 19)$, which finishes the proof.

\section{Proof of the upper bound $F_v(2, 2, 2, 3; 4) \leq 22$}

We need to construct a 22-vertex graph in $\mH_v(2, 2, 2, 3; 4)$. First, we find a large number of 24-vertex and 23-vertex maximal graphs in $\mH_v(2, 2, 2, 3; 4)$.

All 352366 24-vertex graphs $G$ with $\omega(G) \leq 3$ and $\alpha(G) \leq 4$ were found by McKay, Radziszowski and Angeltveit (see \cite{McK_r}). Among these graphs there are 3903 maximal graphs in $\mH_v(2, 2, 2, 3; 4; 24)$.

By removing one vertex from the obtained 24-vertex maximal graphs we find 6 graphs in $\mH_v(2, 2, 2, 3; 4; 23)$. Let us note that the removal of two vertices from the 24-vertex maximal graphs does not produce any graphs in $\mH_v(2, 2, 2, 3; 4; 22)$. Out of the 6 obtained graphs in $\mH_v(2, 2, 2, 3; 4; 23)$, 5 are maximal. One more maximal graph is obtained by adding one edge to the 6th graph, thus we have 6 maximal graphs in $\mH_v(2, 2, 2, 3; 4; 23)$. We find 192 more maximal graphs in $\mH_v(2, 2, 2, 3; 4; 23)$ by applying Procedure \ref{procedure: populate}, which is defined below. The same procedure was used in \cite{BN15a} to find the graph $\Gamma_2 \in \mH_v(4, 5; 7; 17)$.

By removing one vertex from the obtained 23-vertex maximal graphs we find the graph $\Gamma \in \mH_v(2, 2, 2, 3; 4; 22)$, which is displayed on Figure \ref{figure: H(2, 2, 2, 3; 4; 22)}. By removing one edge from the graph $\Gamma$, which is a maximal graph in $\mH_v(2, 2, 2, 3; 4; 22)$, we obtain two more graphs in $\mH_v(2, 2, 2, 3; 4; 22)$.

\begin{figure}[h]
	\centering
	\begin{subfigure}[b]{\textwidth}
		\centering
		\includegraphics[trim={0 0 0 495},clip,height=121px,width=121px]{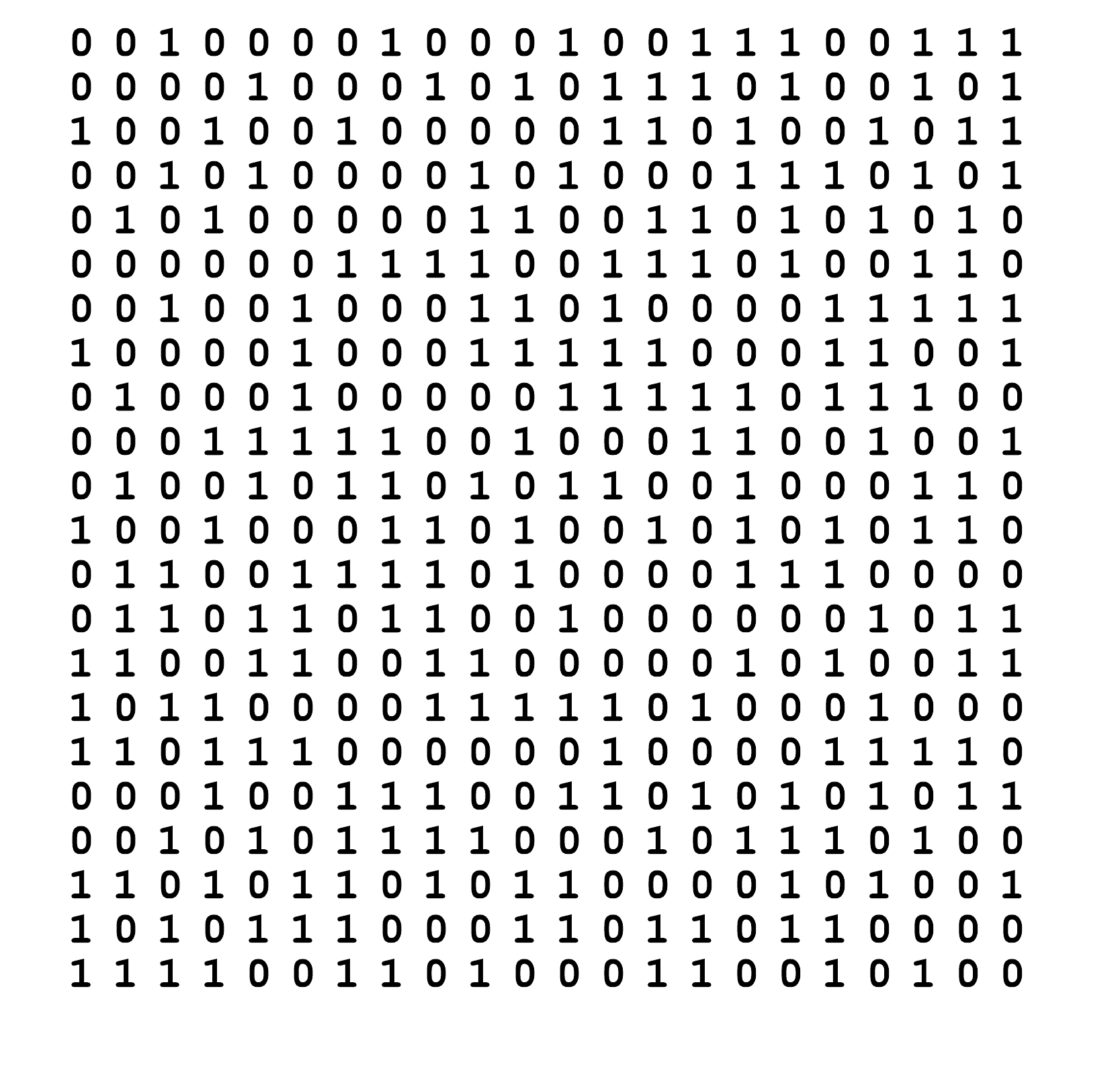}
		\label{figure: G_1}
	\end{subfigure}%
	\vspace{-1.5em}
	\caption{22-vertex graph $\Gamma \in \mH_v(2, 2, 2, 3; 4)$}
	\label{figure: H(2, 2, 2, 3; 4; 22)}
\end{figure}

\begin{procedure}
\label{procedure: populate}
\cite{BN15a}
Extending a set of maximal graphs in $\mH_v(a_1, ..., a_s; q; n)$.

1. Let $\mathcal{A}$ be a set of maximal graphs in $\mH_v(a_1, ..., a_s; q; n)$.

2. By removing edges from the graphs in $\mathcal{A}$, find all their subgraphs which are in $\mH_v(a_1, ..., a_s; q; n)$. This way a set of non-maximal graphs in $\mH_v(a_1, ..., a_s; q; n)$ is obtained.

3. Add edges to the non-maximal graphs to find all their supergraphs which are maximal in $\mH_v(a_1, ..., a_s; q; n)$. Extend the set $\mathcal{A}$ by adding the new maximal graphs.
\end{procedure}

% -----------------------------------------------------------

\section*{Acknowledgements}

The author would like to thank prof. Nedyalko Nenov, who read the manuscript and helped to improve the paper.

% -----------------------------------------------------------

% -----------------------------------------------------------

Aleksandar Bikov

asbikov@fmi.uni-sofia.bg

\smallskip

Faculty of Mathematics and Informatics

Sofia University "St. Kliment Ohridski

5, James Bourchier Blvd.

1164 Sofia, Bulgaria

\vspace{4em}

\begin{center}
\textbf{{\MakeUppercase{Нови граници за върховото число на Фолкман} $F_v(2, 2, 2, 3; 4)$}}\\
\bigskip
\textbf{Александър Биков}\\
\bigskip
\end{center}

За граф $G$ символът $G \overset{v}{\rightarrow} (a_1, ..., a_s)$ означава, че при всяко оцветяване на върховете на $G$ в $s$ цвята съществува $i \in \set{1, ..., s}$, такова че има едноцветна $a_i$-клика от $i$-я цвят. Върховото число на Фолкман $F_v(a_1, ..., a_s; q)$ се дефинира с равенството
$$F_v(a_1, ..., a_s; q) = \min\{\vert\V(G)\vert : G \overset{v}{\rightarrow} (a_1, ..., a_s) \mbox{ и } K_q \not\subseteq G\}.$$
В тази работа ние подобряваме известните граници за числото $F_v(2, 2, 2, 3; 4)$ като доказваме с помощта на компютър, че $20 \leq F_v(2, 2, 2, 3; 4) \leq 22$.

% -----------------------------------------------------------

\end{document}